\documentclass[a4paper, 11pt]{article}
\usepackage{anysize}
\marginsize{3,5cm}{2,5cm}{2,5cm}{2,5cm}%tutaj ustawiasz marginesy
\usepackage{amssymb}
\usepackage{amsmath,amsbsy,amssymb,amscd}
\usepackage{t1enc}
\usepackage[cp1250]{inputenc}
\usepackage[british]{babel}
\usepackage[all]{xy}
\usepackage{color}
\usepackage{amsfonts}
\usepackage{latexsym}
\usepackage{amsthm}
\usepackage{mathrsfs}
\usepackage{hyperref}

\usepackage{color}

\usepackage{changes}

\def\cT{{\mathcal T}}

\definecolor{orange}{rgb}{1,0.5,0}

\usepackage{mathrsfs} %Script dla zapisu \mathscr{}

\DeclareMathAlphabet{\mathpzc}{OT1}{pzc}{L}{it} %stylizowany
\newtheorem{definition}{Definition}[section]

\newtheorem{proposition}[definition]{Proposition}
\newtheorem{theorem}{Theorem}

\newtheorem{corollary}[definition]{Corollary}

\newtheorem{lemma}[definition]{Lemma}

\def\geq{\geqslant}
\def\leq{\leqslant}

\def\eps{\varepsilon}

\def\N{\mathbb{N}}

\def\cB{\mathcal{B}}

\newcommand{\bea}{\begin{eqnarray}}
  \newcommand{\eea}{\end{eqnarray}}
  \newcommand{\beab}{\begin{eqnarray*}}
  \newcommand{\eeab}{\end{eqnarray*}}

  \newcommand{\be}{\begin{equation}}
  \newcommand{\ee}{\end{equation}}

\newcommand{\cQ}{\mathcal Q}

\newcommand{\cC}{\mathcal C}
\newcommand{\cP}{\mathcal P}
\newcommand{\cR}{\mathcal R}

\newcommand{\cL}{\mathcal L}

\title{{Product of two staircase rank one transformations that is not loosely Bernoulli}}
\author{Adam Kanigowski and Thierry de la Rue}

\begin{document}
\bibliographystyle{amsplain}

\baselineskip=14pt \maketitle
\tableofcontents

\begin{abstract}
{
We construct two staircase rank one transformations whose Cartesian product is not loosely Bernoulli.}
\end{abstract}
\section{Introduction}
In this note we study the loosely Bernoulli property for zero entropy,  measure 
preserving transformations. Loose Bernoulliness was introduced by J.\ Feldman, 
\cite{Fel}, and  A.\ Katok, \cite{Kat0}\footnote{In \cite{Kat0} loosely 
Bernoulli is called standard.}.
Recall that a zero entropy measure preserving transformation is loosely Bernoulli (LB for short) if it is
isomorphic to a transformation induced from an irrational rotation of the 
circle. {It follows from  \cite{OrRuWe} that the class of loosely 
Bernoulli systems is broad: it is closed under taking factors, compact 
extensions and inverse limits. Moreover it contains all {\em finite rank 
systems}. First non-loosely Bernoulli examples were constructed by J.~Feldman \cite{Fel}  by a cutting and stacking method. D.~Ornstein, D.~Rudolph and B.~Weiss, \cite{OrRuWe}, constructed a {\em rank one} transformation $T$ such that $T\times T$ is not loosely Bernoulli. The example in \cite{OrRuWe} is based on Ornstein's construction in \cite{Orn} of a class of almost surely mixing rank one transformation with random spacers. On the other hand, M.~Gerber constructed in \cite{Ger} a mixing rank one transformation whose Cartesian square is loosely Bernoulli. An algebraic example of a zero entropy non-loosely Bernoulli transformation comes from the work of M.~Ratner, \cite{Rat1}, where it was shown that the Cartesian square of the
horocycle flow (on compact quotients) is not loosely Bernoulli, although horocycle flows are 
loosely Bernoulli, \cite{Ratner3}.} 

{In this note we study the loosely Bernoulli property for a natural class of mixing rank one transformations: {\em staircase} rank one systems (mixing for this class was proved by Adams, \cite{Adam}). Loose Bernoulliness for products of staircase rank one transformations is not known. Our main result implies in particular that there exit two staircase rank one transformations whose product is not loosely Bernoulli. 
\subsection{Statement of the main result}
 For a rank one system $T$, let $(p^T_n)_{n\in \N}$, $(a_{n,i})_{i=1}^{p_n^T}$ and $(h_n^T)_{n\in \N}$ denote the sequence of {\em cuts}, {\em spacers} and {\em heights} respectively (see Section \ref{sec:ro}).}
{For $0<\gamma<\gamma'<1$ define
\begin{multline}\label{rone}
\cC_{\gamma,\gamma'}:=\{T\in Rank(1)\;:\;\text{ there exists } n'_T\in \N\text{ 
such that for every }\; n\geq n'_T,  \\p^T_n\in[ (h^T_n)^{\gamma}, 
(h^T_n)^{\gamma'}),
a^T_{n,i}>a^T_{n,i-1}\text{ for } i=2,\ldots,p^T_n, \text{ and } a^T_{n,p_n}\leq 
(h^T_n)^{\gamma'} \}.
\end{multline}
Notice that if $T$ is a {\em staircase} rank one transformation, 
i.e. $a^T_{n,i}=i$ for $n\in \N$ and $i\in\{1,\ldots,p^T_n\}$,  and if $p^T_n\in [(h^T_n)^\gamma,(h^T_n)^{\gamma'}]$,
then $T\in C_{\gamma,\gamma'}$.}

{
\begin{definition}[$\delta$-alternating sequences]\label{def:alter}Fix 
$\delta>0$ and let $(a_n)$ and $(b_n)$ be two increasing sequences of positive 
integers. 
We say that $(b_n)$ is  $(a_n)$-{\em alternating with exponent 
$\delta$} if the following holds: there exists $n_0\in \N$ such that for every 
$n\geq n_0$, if $m(n)$ is unique such that $b_{m(n)}\leq a_{n}<b_{m(n)+1}$, then 
$b_{m(n)}^{1+\delta}<a_n$ and  $a_n^{1+\delta}<b_{m(n)+1}$.
Moreover, $(a_n)$ and $(b_n)$ are called {\em $\delta$-alternating}  if  $(a_n)$ 
is $(b_n)$-alternating with exponent $\delta$ and $(b_n)$ is $(a_n)$-alternating 
with exponent $\delta$.
\end{definition}
 For $0<\gamma<\gamma'<1$ let
 \begin{equation}\label{nonlb}
\cL^{\gamma,\gamma',\delta}:=\left\{(T,S)\in \cC_{\gamma,\gamma'}\times 
\cC_{\gamma,\gamma'}\;:\; (h_n^T)\text{ and }(h_n^S) \text{ are } 
\delta-\text{alternating}\right\}.
\end{equation}}
{
\begin{theorem}\label{main} If $T$ is weakly mixing and $(T,S)\in 
\cL^{\gamma,\gamma',\delta}$ with $\delta<\gamma$ then $T\times S$ is not 
loosely Bernoulli.
\end{theorem}
In Subection \ref{sec:next} we will also show the following:\begin{lemma}\label{lem:net} For every $T\in \cC_{\gamma,\gamma'}$, 
$0<\gamma<\gamma'<1/3$ there exists a staircase rank one $S\in \cC_{\gamma/3,3\gamma'}$ such that  
$(T,S)\in \cL^{\gamma/3,3\gamma',\gamma/4}$.
\end{lemma}
It follows from \cite{CreutzSilva2010} that staircase rank one transformations are mixing. Therefore, Theorem \ref{main} together with Lemma \ref{lem:net} has the following consequence:
\begin{corollary}There exists two staircase rank one transformations whose product is not loosely Bernoulli.
\end{corollary}}

\textbf{Acknowledgements:} The authors would like to thank Jean-Paul Thouvenot 
for several discussions on the subject.

\section{Basic definitions}\label{sec:def}

\subsection{Rank one maps}\label{sec:ro}
Recall that a rank one system is constructed by cutting and stacking: fix a 
sequence of {\em cuts } $(p_n)_{n\in \N}$ and a sequence of {\em spacers} 
$(a_{n,i})_{i=1}^{p_n}$, $n\in \N$. We define $h_1=1$ and inductively 
\be\label{eq:height}
h_n=p_{n-1}h_{n-1}+\sum_{i=1}^{p_{n-1}}a_{n-1,i}.
\ee
The sequence $(h_n)_{n\in \N}$ is the sequence of {\em heights}. The rank one 
system $T$ is constructed from the sequences $(p_n)_{n\in \N}$ and 
$(a_{n,i})_{i=1}^{p_n}$, $n\in \N$ in the following way: we start with the 
interval $[0,1]$ which we cut into $p_n$ equal intervals $(I^1_i)_{i=1}^{p_n}$. 
For every $i\in \{1,...,p_n\}$ we put $a_{1,i}$ spacers over $I^1_i$. Then we 
{\em stack} everything over $I^1_1$ and call this tower $\cT_2$ with base 
$I^1_1$. The transformation $T$ on $\cT_2$ just moves one level up except the 
last level, where it is not defined. Next, inductively at stage $n$, we cut the 
tower $\cT_n$ with base $I^n_1$ into $p_n$ equal subtowers, over i'th subtower 
we put $a_{n,i}$ spacers and then we stack to get tower $\cT_{n+1}$ with base 
$I^n_1$. Finally, the rank one map $T$ is defined almost everywhere, i.e. it is 
defined on $\cT_\infty:=\bigcup_{n\in \N} \cT_n$. Often, if there are more 
transformations involved, we will write a superscript $T$ in the sequences 
$(p_n)$, $(a_{n,i})$, $(h_n)$ and $\cT_n$. In what follows we will always assume 
that $p_n>1$ for every $n$. This implies that 
\be\label{eq:expgr}
h_n\geq 2^{n-1}.
\ee
We also assume below that the total measure of all the spacers is finite:
\be\label{eq:finmes}
\sum_{n=1}^{+\infty}\frac{1}{\prod_{k=1}^np_i}\left(\sum_{i=1}^{p_n}a_{n,i}
\right)<+\infty.
\ee
Under this condition, $T$ preserves a probability measure $\mu_T$ given by the 
normalized Lebesgue measure on $\cT_\infty$. It is a classical fact that 
rank one systems are uniquely ergodic. Hence $T$ acts on $(\cT_{\infty}, 
\cB,\mu_T)$. Moreover, \eqref{eq:finmes} implies the following bounds on the 
sequence of heights: there exists $1\le K<\infty$ such that for every $n\in \N$
\be\label{eq:bhe}
\prod_{i=1}^{n-1}p_i\leq h_n\leq K \prod_{i=1}^{n-1}p_i.
\ee
Indeed, the LHS is immediate from \eqref{eq:height}.  The RHS follows by 
\eqref{eq:finmes} and \eqref{eq:height} since for 
$\epsilon_n=\frac{1}{\prod_{i=1}^n p_i}\left(\sum_{i=1}^{p_n}a_{n,i}\right)$, we 
have
$$
\sum_{i=1}^{p_{n-1}}a_{n-1,i} = \epsilon_{n-1}\prod_{i=1}^{n-1}p_i\leq 
\epsilon_{n-1} p_{n-1} h_{n-1}
$$
and so by \eqref{eq:height}
$$
h_n\leq p_{n-1}h_{n-1}(1+\epsilon_{n-1})
$$
and recursively 
$$
h_n\leq \left(\prod_{i=1}^{n-1}(1+\epsilon_i)\right)\prod_{i=1}^{n-1}p_i\leq K 
\prod_{i=1}^{n-1}p_i,
$$
where $K:=\prod_1^\infty(1+\epsilon_i)<\infty$ by~\eqref{eq:finmes}. This shows \eqref{eq:bhe}.
\bigskip

 For $n\in \N$ and $x,y$ in one level of $\cT_{n}$ we define the {\em 
horizontal} distance of $x,y$, setting
\be\label{eq:disthor}
d_H(x,y):=\frac{1}{h_{m(x,y)}},
\ee
where $m=m(x,y)\geq n$ is the largest number such that $x,y$ are in one level of 
$\cT_{m}$ (i.e. they are in different levels of $\cT_{m+1}$). The set of rank one transformations  (satisfying \eqref{eq:finmes}) 
will be denoted $Rank(1)$.

\subsection{Loosely Bernoulli transformations}
We recall the definition of $\bar{f}$ metric introduced in \cite{Fel}. For two 
finite words (over a finite alphabet) $A=a_1...a_k$ and $B=b_1...b_k$ a {\em 
matching} between $A$ and $B$ is 

any pair of strictly increasing sequences $(i_s,j_s)_{s=1}^r$ such that 
$a_{i_s}=b_{j_s}$ for $s=1,...r$. The $\bar{f}$ distance between $A$ and $B$ is 
defined by 
$$
\bar{f}(A,B)=1-\frac{r}{k},
$$
where $r$ is the maximal cardinality over all matchings between $A$ and $B$.

Let $T:(X,\cB,\mu)\to (X,\cB,\mu)$ be a measure preserving automorphism. For a 
finite partition $\cP=(P_1,...P_r)$ of $X$ and an integer $N\geq 1$ we denote 
$\cP_0^{N}(x)=x_0...x_{N-1}$, where 
$x_i\in\{1,...,r\}$ is such that $T^i(x)\in P_{x_i}$ for $i=0,...,N-1$.

\begin{definition}\label{def:LB} The process $(T,\cP)$ is said to be {\em 
loosely Bernoulli} if for every $\eps>0$ there exists $N_\eps\in \N$ and a set 
$A_\eps\in \cB$, $\mu(A_\eps)>1-\eps$ such that for every $x,y\in A_\eps$ and 
every $N\geq N_\eps$
\be\label{eq:leqeps}
\bar{f}(\cP_0^N(x),\cP_0^N(y))<\eps.
\ee
The automorphism $T$ is {\em loosely Bernoulli} (LB for short) if for every 
finite measurable partition $\cP$, the process $(T,\cP)$ is loosely Bernoulli. 
\end{definition}

To prove Theorem \ref{main}, in view of Definition \ref{def:LB} it is enough to 
show that there exists a finite partition $\cP$ such that the process $(T\times 
S,\cP)$ is not LB. Recall that we have an exhaustive sequence of towers 
$(\cT_n^T)$ and $(\cT_n^S)$ for $T$ and $S$ respectively. For $n\in \N$ let  
$\cQ_n$ be the partition of $\cT^T_\infty$ into levels of $\cT_n^T$ and 
$\cT^T_{\infty}\setminus \cT_n^T$ ($\cT_\infty^T\setminus \cT^T_n$ is the union 
of all spacers at stages $\geq n$). Let $\cR_n$ be the analogous partition of 
$\cT^S_\infty$ and define $\cP_n:=\cQ_n\times \cR_n$. Then $\cP_n$ is a 
partition of $\cT^T_\infty\times \cT^S_\infty$ (in fact, the sequence $(\cP_n)$ 
converges to partition into points). Theorem \ref{main} follows by the following 
theorem:
{
\begin{theorem}\label{main2} Let $T$ be weakly mixing and $(T,S)\in \mathcal{L}^{\gamma,\gamma',\delta}$ with $\delta<\gamma$. There exists $m_0$ such that the process  $(T\times 
S,\cP_{m_0})$ is not LB.
\end{theorem}
}

In the following subsections we will state some lemmas which will help us in 
proving Theorem \ref{main2}.
\section{Preliminary lemmas}
\subsection{A combinatorial lemma: lower bound on $\bar{f}$}
\label{sec:BegEnd}
In this section we state a general combinatorial lemma which allows to give a 
lower bound on the $\bar{f}$ distance. Let $(i_s)_{s=1}^r$ and $(j_s)_{s=1}^r$ 
be two increasing sequences of positive integers in $[0,N]$. 

For $M\leq N$ and $1\leq w<r$ define
$$
I(M,w)=\{s\in\{1,\ldots r\}\;:\; i_s\in [i_w,i_w+M]\}\subset [0,N]
$$
and
$$
J(M,w)=\{s\in\{1,\ldots r\}\;:\; j_s\in [j_w,j_w+M]\}\subset [0,N].
$$ 
The following property of $I(M,w)$ and $J(M,w)$ is straightforward and will 
be useful in the proof of the lemma below:
{If $s>w$ and} $s\notin I(M,w)$, then 
\be\label{begemptyset}
I(M,w)\cap I(M,s)=\emptyset,
\ee
and analogously for $J(M,w)$.
We have the following lemma:

\begin{lemma}\label{comb} Let $\xi\in(0,1)$. Let $(i_s)_{s=1}^r\in [0,N]^r$ and 
$(j_s)_{s=1}^r\in [0,N]^r$ be two increasing sequences of integers for which 
there exists a number $K\in \N$, $8K^{1+\xi}\leq N$ such that for every 
$s=1,\ldots,r$
\begin{equation}\label{beend}
\min\left(|I(K^{1+\xi},s)|,|J(K^{1+\xi},s)|\right)\leq 2K,
\end{equation}
then $r<\frac{4N}{K^{\xi}}$.
%then $r<\frac{8N}{K^{\xi}}$.
\end{lemma}

\begin{proof}
Set $s_1:=1$. 
If $|I(K^{1+\xi},s_1)|\geq |J(K^{1+\xi},s_1)|$, we set 
$B_1:=J(K^{1+\xi},s_1)$ otherwise,  we set $B_1:=I(K^{1+\xi},s_1)$. By \eqref{beend}, we have  
$$
|B_1|\leq 2K.
$$
We then proceed inductively: assume that for some $\ell\ge1$
we have constructed $1=s_1<s_2<\cdots<s_\ell\le r$ and subsets $B_1,\ldots,B_\ell$ of $\{1,\dots,r\}$.
If $B_1\cup\cdots\cup B_\ell=\{1,\dots,r\}$, we stop here and set $v:=\ell$. Otherwise, we define 
$s_{\ell+1}$ as the smallest element of $\{1,\dots,r\}\setminus B_1\cup\cdots\cup B_\ell$, and $B_{\ell+1}$ as
either $I(K^{1+\xi},s_{\ell+1})$ or $J(K^{1+\xi},s_{\ell+1})$ (by choosing the set with the smallest cardinal).
By \eqref{beend}, we always have  
$$
|B_{\ell+1}|\leq 2K.
$$

We go on in this way until we obtain a finite sequence of set $\{B_i\}_{i=1}^v$, such that 
\be\label{eq:cover}
B_1\cup\cdots\cup B_v=\{1,\dots,r\},
\ee
and such that 
\be\label{eq:card}
|B_\ell|=\min(|I(K^{1+\xi},s_\ell)|,|J(K^{1+\xi},s_\ell)|)\leq 2K\text{ for every } 
\ell\in\{1,\ldots,v\}.
\ee

% Since 
% $8K^{1+\xi}\leq N$ we are able to make at least one step of the construction 
% (i.e. $v\geq 2$). By \eqref{beend} for $s=s_i$, $i=1,\ldots, v$ it follows that
% that
% \be\label{eq:card}
% |B_i|=\min(|I(K^{1+\xi},s_i)|,|J(K^{1+\xi},s_i)|)\leq 2K\text{ for every } 
% i\in\{1,\ldots,v\}.
% \ee
Moreover, if $1\le \ell_1<\ell_2 \le v$, then $s_{\ell_2}\notin B_{\ell_1}$. If we further
assume that $B_{\ell_h}=I(K^{1+\xi},s_{\ell_h})$ for $h\in\{1,2\}$ (or 
$B_{\ell_h}=J(K^{1+\xi},s_{\ell_h})$ for $h\in\{1,2\}$), then by construction (see also 
\eqref{begemptyset})
\be\label{newemp}
B_{\ell_1}\cap B_{\ell_2}=\emptyset.
\ee
By pigeonhole principle, there exists a set $A\subset 
\{1,\ldots,v\}$, with $|A|\geq \frac{v}{2}$ and such that either for every $\ell\in A$, 
$B_\ell=I(K^{1+\xi},s_\ell)$, or for every $\ell\in A$, $B_\ell=J(K^{1+\xi},s_\ell)$. 
By \eqref{newemp}, the subsets $B_\ell$, $\ell\in A$ correspond to disjoint subintervals of $[0,N]$,
each of them of length $K^{1+\xi}$. We thus get 
\begin{equation}\label{smv}
\frac{v}{2} \leq \frac{N}{K^{1+\xi}}.
\end{equation}
Now, \eqref{eq:cover}, \eqref{eq:card} and \eqref{smv} yield
\[
 r \le |B_1|+\cdots+|B_v| \le 2Kv \le \frac{4N}{K^\xi}.
\]
\end{proof}

\subsection{Proof of Lemma \ref{lem:net}}\label{sec:next}

%\begin{lemma}\label{lem:WM}
%Let $T\in \cC_{\gamma,\gamma'}$ be a staircase rank one transformation. Then $T$ is weakly mixing.
%\end{lemma}

%\begin{proof}[Proof of Lemma \ref{lem:WM}:]
% First, we use the classical 
%result that a staircase rank one transformation is always mixing on the sequence $(h_n)$, i.e. for all $f,g\in L^2(\mu_T)$, 
%$\int f   (g\circ T^{h_n}) \, d\mu_T \to (\int f\, d\mu_T) (\int g\, d\mu_T)$. Indeed, it is enough to prove this when 
%$f$ and $g$ are characteristic functions of finite unions of levels of some tower in the construction, and in this case 
%this follows from the structure of the spacers above the towers, and the fact that, by ergodicity, 
%\[
 %\frac{1}{p_n}\sum_{i=1}^{p_n} \int f (g\circ T^i) \, d\mu_T \to \left(\int f\, d\mu_T\right) \left(\int g\, d\mu_T\right).
%\]
%(See Theorem~3 in \cite{CreutzSilva2004} for details.) 
%But then, if  $\lambda$ is an eigenvalue of $T$ and if $f$ is a corresponding eigenvector, we have on the one hand
%\[
 %\left|\int f \bar f\circ T^{h_n} \, d\mu_T\right| = |\lambda^{h_n}| \int |f|^2  \, d\mu_T\ = \int |f|^2  \, d\mu_T,
%\]
%and on the other hand,
%\[
 %\left|\int f \bar f\circ T^{h_n} \, d\mu_T\right| \to \left| \int f\, d\mu_T \right|^2.
%\]
%It follows that $f$ is constant and $\lambda=1$. This proves that $T$ is weakly mixing.
%\end{proof}

%Now we will prove Lemma \ref{lem:net}:

\begin{proof}[Proof of Lemma \ref{lem:net}:] Let $(h_n^T)$ denote the sequence 
of heights for $T$. We will construct a {staircase }$S\in C_{\gamma/3,3\gamma'}$ such that 
$(T,S)\in \cL^{\gamma/3,3\gamma',\gamma/4}$, \textit{i.e.} the sequences $(h_n^T)$ and 
$(h_n^S)$ are $\gamma/4$- alternating. %We will define the sequences $(p_n^S)$and $(a^S_{n,i})_{i=1}$, $n\in \N$. It is enough to show that 
For this, we will show that there exists $n_0$ 
such that for every $n\geq n_0$, we have
\be\label{eq:heightsinde}
(h_{n-1}^S)^{1+\gamma/4}<h_n^T\;\;\text{ and }\;\; 
(h_{n}^T)^{1+\gamma/4}<h_{n}^S.
\ee

Fix $\eta<1/100$. Let $K_T$ denote the constant $K$ for $T$ in \eqref{eq:bhe}. We will first construct $(p_i^S)_{i=1}^{+\infty}$ so that for $n\geq n_0$, we have 
\be\label{eq:conp}
\left(\prod_{i=1}^{n-2}p_i^S\right)^{1+(\frac{1}{4}+\eta)\gamma}<\prod_{i=1}^{n-1}
p_i^T\;\;\text{ and }\;\; \left(K_T 
\prod_{i=1}^{n-1}p_i^T\right)^{1+\gamma/4}<\prod_{i=1}^{n-1}p_i^S.
\ee
Then we will show how to derive~\eqref{eq:heightsinde} from~\eqref{eq:conp}.

Since $T\in \cC_{\gamma,\gamma'}$ it follows that for $n > n'_T$, by 
\eqref{eq:bhe}, we have 
\be\label{eq:enT}
K_T\prod_{i=1}^{n}p_i^T\geq h_{n}^T\geq (h_{n-1}^T)^{1+\gamma}\geq 
(\prod_{i=1}^{n-1}p_i^T)^{1+\gamma}.
\ee

Set $p^S_n=2$ for every $n\in\{1,...,n'_T+1\}$ and $a_{n,i}^S=0$ for every 
$i\in\{1,2\}$ and $n\in\{1,...,n'_T+1\}$. Then (by enlarging $n'_T$ if 
neccesary), we have
$$
\left(\prod_{i=1}^{n'_T+1}p_i^S\right)^{1+(\frac{1}{4}+\eta)\gamma}=( 
2^{n'_T+1})^{1+(\frac{1}{4}+\eta)\gamma}\leq \left(\frac{1}{K_T}2^{n'_T+1}\right)^{1+\gamma}\leq 
\prod_{i=1}^{n'_T+2}p_i^T.
$$
(The last inequality by \eqref{eq:enT} and since $p_n^T\geq 2$.) So the left
inequality in \eqref{eq:conp} holds for $n=n'_T+3$. We then proceed 
inductively: having defined $(p_{i}^S)_{i=1}^{w}$ so that the left inequality 
in \eqref{eq:conp} holds for $n=w+2$, we choose $p_{w+1}^S\in \N$ so that
\be\label{eq:loint}
\left(\prod_{i=1}^{w}p_i^S\right)p_{w+1}^S\in \left[\left(K_T 
\prod_{i=1}^{w+1}p_i^T\right)^{1+\gamma/4}, 
\left(\prod_{i=1}^{w+2}p_i^T\right)^{\frac{1}{1+(\frac{1}{4}+\eta)\gamma}}\right].
\ee
We explain below why such a choice is always possible.
Notice that by \eqref{eq:enT} for $n=w+2$, we have 
\begin{multline*}
\left| \left(\prod_{i=1}^{w+2}p_i^T\right)^{\frac{1}{1+(\frac{1}{4}+\eta)\gamma/4}}-\left(K_T 
\prod_{i=1}^{w+1}p_i^T\right)^{1+\gamma/4}\right|\geq \\\left|
\left(\frac{1}{K_T}\prod_{i=1}^{w+1}p_i^T\right)^{\frac{1+\gamma}{1+(\frac{1}{4}+\eta)\gamma}}-\left(K_T 
\prod_{i=1}^{w+1}p_i^T\right)^{1+\gamma/4}\right|\geq 4\prod_{i=1}^{w+1}p_i^T,
\end{multline*}
the last inequality by $\frac{1+\gamma}{1+(\frac{1}{4}+\eta)\gamma}\geq 1+\gamma/4$ (recall also that $K_T$ is a fixed constant and $w\geq n'_T+3$ and $n'_T$ can be made sufficiently large with respect to $K_T$.) Moreover, since we assume that the left inequality in \eqref{eq:conp} holds for $w+2$, it follows that $4\prod_{i=1}^{w+1}p_i^T\geq  4\prod_{i=1}^{w}p_i^S$. Therefore the length of the interval on the right of \eqref{eq:loint} is at least $4\prod_{i=1}^{w}p_i^S$, and so such $p_{w+1}^S\geq 2$ always exists. Recursively, it follows that \eqref{eq:conp} holds for $n\ge n_0:=n'_T+3$. 

{To guarantee that $S$ is a staircase rank one system we set $a^S_{n,i}=i$ for $n\in \N$ and $i\in\{1,..., p_n\}$}. It remains to show that $S\in \cC_{\gamma/3,3\gamma'}$ and, using \eqref{eq:conp}, that \eqref{eq:heightsinde} holds.
Notice that by \eqref{eq:conp}, for $w
\geq n_0$, we have 
\be\label{eq:newqw}
\frac{p^S_w}{\prod_{i=1}^{w-1}p^S_i}\leq \frac{\prod_{i=1}^{w+1}p^T_i}{\left(\prod_{i=1}^{w-1}p^S_i)\right)^2}\leq
 \frac{\prod_{i=1}^{w+1}p^T_i}{\left(\prod_{i=1}^{w-1}p^T_i\right)^2}=\frac{p_w^Tp_{w+1}^T}{\prod_{i=1}^{w-1}p^T_i}.
\ee
Since $T\in \cC_{\gamma,\gamma'}$ with $0<\gamma<\gamma'<1/3$ and by \eqref{eq:bhe}, we get   
\be
\label{eq:newqw2}
\frac{p_w^Tp_{w+1}^T}{\prod_{i=1}^{w-1}p^T_i}
\leq \frac{h_w^{\gamma'}h_{w+1}^{\gamma'}}{\prod_{i=1}^{w-1}p^T_i}
\leq \frac{K_T^2 \left(\prod_{i=1}^{w-1}p^T_i\right)^{2\gamma'}(p^T_w)^{\gamma'}}{\prod_{i=1}^{w-1}p^T_i}
\leq \frac{K_T^3 \left(\prod_{i=1}^{w-1}p^T_i\right)^{3\gamma'}}{\prod_{i=1}^{w-1}p^T_i}.
\ee
Notice that \eqref{eq:newqw} and \eqref{eq:newqw2} together with the definition of spacers for $S$ and $\gamma'<1/3$ imply that 
\begin{align*}
\sum_{n=1}^{+\infty}\frac{1}{\prod_{i=1}^np^S_i}\left(\sum_{i=1}^{p_n}a^S_{n,i}\right)
& \leq \sum_{n=1}^{+\infty}\frac{(p_n^S)^2}{\prod_{i=1}^{n}p^S_i} \\
& \leq \sum_{n=1}^{+\infty}\frac{p_n^S}{\prod_{i=1}^{n-1}p^S_i}<+\infty,
\end{align*}
and hence \eqref{eq:bhe} holds for $S$ (with constant $K_S$). Then, \eqref{eq:conp} and \eqref{eq:bhe} (for $T$ and for $S$) ensure the validity of \eqref{eq:heightsinde} for $n$ large enough. This proves that $(h_n^T)$ and $(h_n^S)$ are $\gamma/4$- alternating.

Now it remains to check that $S\in C_{\gamma/3,3\gamma'}$.
For this, we will use the following inequality: for $n$ large enough, using several times \eqref{eq:bhe} and the fact that $T\in C_{\gamma,\gamma'}$, we get
\begin{align*}
 h_{n+2}^T 
 & \le K_T \prod_{i=1}^{n+1} p_i^T \\ %\quad\text{by \eqref{eq:bhe}}
 & = K_T \, p_{n+1}^T \,  p_n^T \prod_{i=1}^{n-1} p_i^T \\
 & \le K_T \, p_{n+1}^T \,  p_n^T \, h_n^T \\ %\quad\text{again by \eqref{eq:bhe}}
 & \le K_T \, (h_{n+1}^T)^{\gamma'} \, (h_n^T)^{1+\gamma'} \\ %\quad\text{as }T\in C_{\gamma,\gamma'} 
 & \le  K_T  \,( K_T \, p_n^T\, h_n^T) ^{\gamma'} \, (h_n^T)^{1+\gamma'} \\
 & \le K_T^2  \, \bigl((h_n^T)^{1+\gamma'}\bigr)^{\gamma'} \, (h_n^T)^{1+\gamma'} \\
 & = K_T^2  \, (h_n^T)^{1+2\gamma'+\gamma'^2} \\
 & \le (h_n^T)^{1+3\gamma'}.
\end{align*}
By  \eqref{eq:heightsinde}, it follows that for $n$ large enough, we have
$$
\frac{h_{n+1}^S}{h_n^S}
\leq \frac{(h_{n+2}^T)^{\frac{1}{1+\gamma/4}}}{(h_n^T)^{1+\gamma/4}}
\leq \frac{h_{n+2}^T}{h_n^T}
\leq (h_n^T)^{3\gamma'}\leq (h_n^S)^{3\gamma'}.
$$
On the other hand, again by \eqref{eq:heightsinde}
$$
\frac{h_{n+1}^S}{h_n^S}\geq 
\frac{(h_{n+1}^T)^{1+\gamma/4}}{(h_{n+1}^T)^{\frac{1}{1+\gamma/4}}}\geq 
(h_{n+1}^T)^{\frac{(1+\gamma/4)^2-1}{1+\gamma/4}}\geq (h_{n}^S)^{\gamma/2}.
$$
Therefore and by \eqref{eq:bhe}, for $n$ large enough, we have 
\be\label{eq:pns}
p_n^S\in \left[\frac{h_{n+1}^S}{h_n^S}, \frac{K_S \, h_{n+1}^S}{ h_n^S}\right]\subset 
 \left[(h_n^S)^{\gamma/3},(h_n^S)^{3\gamma'}\right].
\ee
Using \eqref{eq:pns} and the definition of spacers for $S$ ($S$ is a staircase transformation) we get that 
$S\in \cC_{\gamma/3,3\gamma'}$.  This finishes the proof.
\end{proof}

\subsection{Distribution of points for maps from 
$\cC_{\gamma,\gamma'}$}\label{sec:distest}
In this section we will do quantitative estimates on recurrence of points. Fix 
$G\in \cC(\gamma,\gamma')$. Since $G$ will be fixed throughout this section, we 
drop the superscript $G$ and denote by $\cT_n$, $B_n$, $h_n$ 
and $p_n$ the tower, base, height and the number of cuts at stage $n$ (recall 
that $p_n\in [h_n^{\gamma},h_n^{\gamma'}]$ for $n\geq n'_G$). For $i\in 
\{1,\ldots, p_n\}$, let $\cT_{n,i}\subset \cT_n$ denote the column over the 
$i$-th cut.  

We will construct some subsets of $\cT_\infty$, on which we control the dynamics 
well. First, we cut off the ``boundaries'' of $\cT_n$: let 
\be\label{eq:fng}
F_n^G:=\left(\bigcup_{i=\left\lfloor\frac{h_n}{n^2}\right\rfloor}^{h_n-\left\lfloor\frac{h_n}{n^2}\right\rfloor}G^i(B_n)\right)
\cap\left(\bigcup_{i=h_n^{\gamma/4}}^{p_n-h_n^{\gamma/4}}\cT_{n,i}\right).
\ee

Notice that since $p_n\geq h_n^\gamma$ for $n\geq n'_G$ and every column has 
equal measure, we have $\mu_G(F_n^G)\geq 1-\frac{4}{n^2}$ for $n$ large enough. Let
\begin{equation}\label{fg}
F^G:=\bigcap_{i\geq n_1}F_i^G,
\end{equation}
where $n_1$ is such that $\mu(F^G)\geq 1-10^{-5}$. In particular, if $x\in F^G$, 
then $x\in \cT_{n_1}$. This will be used in the statement of the lemma below. 
%(see \eqref{eq:disthor} for $n=n_1$).

\begin{lemma}\label{lem:1}Fix $0<\xi\le 100^{-1}\min(\gamma,1-\gamma')$. There 
exists $m_1\in \N$ such that for every $n\geq m_1$ the following holds: for 
every $x,x'\in F^G$ such that $d_H(x,x')=\frac{1}{h_n}$ (see 
\eqref{eq:disthor}), for every $r\in\{h_n,h_{n}+1,\ldots,(h_n)^{1+2\xi}\}$ for 
which $G^rx,G^rx'\in F^G$, we have
\begin{equation}\label{short:bl}
d_H(G^r(x),G^r(x'))\geq \frac{1}{h_{n-1}}.
\end{equation}
Moreover, for every $x,x'\in F^G$ such that $d_H(x,x')\leq \frac{1}{h_n}$ and 
for every $i,j\in\{0,...\frac{h_n}{n^3}\}$, $i\neq j$, we have
\begin{equation}\label{unif:dist}
d_H(G^i(x),G^j(x'))\geq \frac{1}{h_{n-1}}.
\end{equation}
\end{lemma}

\begin{proof} 
% Notice that for every $n\in \N$, every $x\in \cT_n$ and every 
% $0< k< h_n$, $x$ and $G^kx$ are not in the same level of $\cT_n$. In 
% particular, for every $0< k<h_n$ for which $G^kx\in F^G\subset \cT_{n_1}$,  we 
% have
% \be\label{eq:low}
% d_H(x,G^kx)\geq \frac{1}{h_{n-1}}.
% \ee
Let $n\ge n_1$ be large enough so that $h_n^{\gamma'+2\xi}< h_n/n^2$, and let 
$x,x'\in F^G$ be such that $d_H(x,x')=\frac{1}{h_n}$. 
Fix $r\in\{h_n,h_{n}+1,\ldots,h_n^{1+2\xi}\}$ for which $G^rx,G^rx'\in F^G$. Let 
$0\leq \ell_1,\ell_2\leq p_n$ be such that $x\in \cT_{n,\ell_1}$ and $x'\in 
\cT_{n,\ell_2}$. Since $x,x'\in F^G\subset F_n^G$ (see \eqref{eq:fng}), $\ell_1,\ell_2\in[h_n^{\gamma/4},p_n-h_n^{\gamma/4}]$. Notice that 
$r\leq h_n^{1+2\xi}$ and the height of $\cT_n$ is $h_n$. By the choice of 
$\xi$ it follows that
\be\label{eq:notf}
 G^rx,G^rx'\notin \bigcup_{i=p_n-\frac{1}{2}h_n^{\gamma/4}}^{p_n}\cT_{n,i}.
\ee
Since $d_H(x,x')=\frac{1}{h_n}$, we have $\ell_1\neq \ell_2$ (otherwise 
$d_H(x,x')\leq \frac{1}{h_{n+1}}$). Assume WLOG that $\ell_1< \ell_2$. 
% We claim that there exists $0<w_r< h_n/2$ %(total displacement of spacers) 
% such that
% \be\label{eq:totdis}
% d_H(G^rx,G^{r+w_r}x')\leq \frac{1}{h_n}.
% \ee
% Indeed, 
Let $\ell_r\geq \ell_1$ be such that $G^rx\in \cT_{n,\ell_r}$. Notice 
that since $r\geq h_n$ and $G^rx\in F^G\subset \cT_n$, it follows that in fact 
$\ell_r>\ell_1$. Define
$$
w_r:=\sum_{i=\ell_1}^{\ell_r-1}(a_{n,\ell_2-\ell_1+i}-a_{n,i}).
$$
By $G\in C_{\gamma,\gamma'}$, using \eqref{eq:notf} we get that 
$a_{n,\ell_2-\ell_1+i}-a_{n, i}>0$ (the spacers are monotonically placed). Since 
$\ell_r>\ell_1$ it follows that $w_r>0$. Moreover, $a_{n,i}\leq h_n^{\gamma'}$ 
for $i\in \{1,\ldots,p_n\}$, hence 
$$
w_r\leq (\ell_r-\ell_1)h_n^{\gamma'}\leq 
\frac{h_n^{1+2\xi}}{h_n}h_n^{\gamma'}=h_n^{\gamma'+2\xi}< h_n/n^2.
$$

Let $s$ be such that $G^rx\in G^s B_n$. 
Since $G^rx\in F^G\subset F_n^G$, we have $s\ge h_n/n^2$, and it follows by definition of $w_r$ 
and the fact that $w_r<s$ that $G^rx'\in G^{s-w_r}B_n$. Hence $G^rx$ and $G^{r}x'$ are not 
in the same level of $\cT_n$ and \eqref{short:bl} follows.

\smallskip
For \eqref{unif:dist}, notice that $x,x'\in F^G\subset F_n^G$ (see 
\eqref{eq:fng}) and by assumptions $x,x'$ are in one level of $\cT_n$, i.e. for 
some $s\in [\frac{h_n}{n^2},h_n-\frac{h_n}{n^2}]$, $x,x'\in G^s(B_n)$. But by 
the bounds on $s$ it follows that for $u\in \{i,j\}$, we have $G^u(x)\in 
G^{s+u}(B_n)$ (and $0\leq s+u\leq h_n$). Since $i\neq j$, we have $s+i\neq s+j$ 
and hence $G^i(x)$ and $G^j(x')$ are in different levels of $\cT_n$. This gives 
\eqref{unif:dist}. 
\end{proof}

For $n\in \N$ and $x\in \cT_n$ let $i_{n,x}\in\{0,\ldots, h_n-1\}$ be such that 
$x\in G^{i_{n,x}}(B_n)$. Define
\begin{equation}\label{dgn}
D_G^{x,n}:=\cT_n\setminus 
\bigcup_{k=-\left\lfloor\frac{h_n}{n^2}\right\rfloor}^{\left\lfloor\frac{h_n}{n^2}\right\rfloor}G^{k+i_{n,x}}(B_n)
\end{equation}
%  In other words, for every $-\frac{h_n}{n^2}\leq k\leq\frac{h_n}{n^2}$, $T^kz$ 
% and $x$ are not in one level of $\cT_n$. 
Notice that, since $G\in C_{\gamma,\gamma'}$, we have $\mu(\cT_n) \ge 1-{\rm 
O}(n^{-2})$, hence $\mu(D_G^{x,n})\geq 1-{\rm 
O}(n^{-2})$.
We define
\begin{equation}\label{dx}
D^G_x:=\bigcap_{n\geq n_3}D_G^{x,n},
\end{equation}
where $n_3$ is such that for all $x$, $\mu(D^G_x)\geq 1-10^{-5}$.

\begin{lemma}\label{lem:2}There exists $m_2\in \N$ such that for every $x\in 
F^G$ (see \eqref{fg}), every $x'\in D^G_x$ every $n\geq m_2$  and  every 
$i,j\in\left\{0,...,\left\lfloor\frac{h_{n+1}}{(n+1)^2}\right\rfloor\right\}$ for which $G^ix\in \cT_n$, we have 
\begin{equation}\label{ffo}
d_H(G^i(x),G^j(x'))\geq \frac{1}{h_n}.
\end{equation}
\end{lemma}
\begin{proof}Let $m_2:=2\max(n_1,n_3)$ and fix $n\geq m_2$. Let $x$, 
$x'$ and $i,j$ be as in the assumptions of the lemma. 
Assume by contradiction that  $d_H(G^i(x), G^j(x'))<\frac{1}{h_n}$. 
By \eqref{eq:disthor} this implies that $G^i(x)$ and 
$G^j(x')$ are in the same level of $\cT_{n+1}$. Since $x\in F^G\subset F_{n+1}^G$ 
(see \eqref{eq:fng}) and $0\le i\le\frac{h_{n+1}}{(n+1)^2}$ , this implies that $G^ix$ and $G^j(x')$ are both located $i$ levels above $x$ in $\cT_{n+1}$.
Therefore $x$  and $G^{j-i}x'$ are in the same level of $\cT_{n+1}$. But $|j-i|\leq 
\frac{h_{n+1}}{(n+1)^2}$, which contradicts the fact that $x'\in 
D^G_x\subset D_G^{x,n+1}$ (see \eqref{dgn}). The contradiction finishes the 
proof.
\end{proof}

\section{A proposition which implies Theorem \ref{main2}}
For the rest of the paper $(T,S)\in \cL^{\gamma,\gamma',\delta}$ with 
$0<\delta<\gamma<\gamma'<1$ are fixed. Since we deal with two transformations 
$T$ and $S$, we will denote the horizontal distance (see \eqref{eq:disthor}) for 
$T$ and $S$ respectively by $d_1$ and $d_2$. 

Recall from \eqref{fg} the sets $F^T$ and $F^S$. Let $\cP=\cP_n$ for some $n\ge 3$. For $(x,y),(x',y')\in 
\cT_\infty^T\times\cT_\infty^S$, and $N\in \N$, consider a matching $\theta=(i_s,j_s)_{s=1}^r$ of $\cP_0^N(x,y)$ and $\cP_0^N(x',y')$. Then for each $k\in \N$ we define 
\begin{multline}\label{eq:ats}
%
% A_{n}^{j,N}((x,y)(x',y')):=\{r \in [0,N]\;:\; \\(T^{i_r}\times 
% S^{i_r})(x,y),(T^{j_r}\times S^{j_r})(x',y')\in (F^T\cap \cT_n^T)\times (F^S\cap 
% \cT_n^S) \text{ and }\\
% 2^{-j-1}\leq 
% \max\left(d_1(T^{i_r}x,T^{j_r}x'),d_2(S^{i_r}y,,S^{j_r}y')\right)<2^{-j}\}.
%
%
A_{\theta}^{k}((x,y)(x',y')):=\Bigl\{s \in \{1,\ldots,r\}\;:\; \\(T^{i_s}\times 
S^{i_s})(x,y),(T^{j_s}\times S^{j_s})(x',y')\in (F^T\cap \cT_n^T)\times (F^S\cap 
\cT_n^S) \text{ and }\\
2^{-k-1}\leq 
\max\left(d_1(T^{i_s}x,T^{j_s}x'),d_2(S^{i_s}y,,S^{j_s}y')\right)<2^{-k}\Bigr\}.
\end{multline}
%One could take any geometric sequence instead $(2^{-j})_{j\in \N}$. 
Notice that by definition of the partition $\cP_n$ (every level of 
$\cT_n^T\times \cT_n^S$ is a different atom), the definition of the horizontal 
distance and \eqref{eq:expgr},  we have for each $s\in \{1,\ldots,r\}$
$$
 \max\left(d_1(T^{i_s}x,T^{j_s}x'),d_2(S^{i_s}y,,S^{j_s}y')\right)\le \frac{1}{h_n}
 \le \frac{1}{2^{n-1}}<\frac{1}{2^{n/2}}
$$
Hence for 
%$j<\frac{\log n}{2}$ we have
$k<\frac{n}{2}$, we have
\begin{equation}\label{Aempty}
A_{\theta}^{k}((x,y)(x',y'))=\emptyset.
\end{equation}
The following proposition implies Theorem \ref{main2}:
\begin{proposition}\label{prop:LB} 
There exist  $n_0\geq {100}$,  $N_0\in \N$ 
and a set $B\times C\subset \cT_\infty^T\times \cT_\infty^S$, 
$(\mu_T\times\mu_S)(B\times C)\geq 99/100$, for which the following holds: for 
every $(x,y)\in B\times C$ there exists a set $D_x\times D_y\in 
\cT_\infty^T\times \cT_\infty^S$, $(\mu_T\times\mu_S)(D_x\times D_y)\geq 99/100$ 
such that for every $(x,y)\in B\times C$, $(x',y')\in D_x\times D_y$
and every $N\geq N_0$ 
\begin{enumerate}
\item[$(P1)$] for each $(z,w)\in\{(x,y),(x',y')\}$, 
$$\left|\{i\in [0,N]\;:\; (T^iz,S^iw)\in (\cT^T_{n_0}\cap 
F^T)\times (\cT^S_{n_0}\cap F^S)\}\right| \geq \frac{9}{10}N;$$
\item [$(P2)$]for every $N>N_0$, every $k\in\N$ and every matching $\theta=(i_s,j_s)_{s=1}^r$ of $(\cP_{n_0})_0^N(x,y)$ and $(\cP_{n_0})_0^N(x',y')$, we have 
$$
\left|A_{\theta}^{k}((x,y)(x',y'))\right|\leq \frac{N}{k^2}.
$$
\end{enumerate}
\end{proposition}
The proof of Proposition \ref{prop:LB} is the most technical part of the paper. 
We will devote a separate section to its proof. Let us first show how 
Proposition \ref{prop:LB} implies Theorem \ref{main2}.
\begin{proof}[Proof of Theorem \ref{main2}] To simplify the notation, set 
$\cP:=\cP_{n_0}$. 
We will prove the following: \\
\textbf{Claim:} For every $(x,y)\in B\times C$, $(x',y')\in D_x\times D_y$ and 
every $N\geq N_0$, we have 
\begin{equation}\label{flarge}
\bar{f}\left(\cP^N_0(x,y),\cP^N_0(x',y')\right)\geq 1/100.
\end{equation}
Before we show the \textbf{Claim}, let us show how it implies the result.  
Assume by contradiction that the process $(T\times S,\cP)$ is 
LB. Let $\epsilon:=\frac{1}{200}$ and let $N_\epsilon\in \N$, $A_\epsilon\subset 
\cT_\infty^T\times\cT_\infty^S$, $(\mu_T\times\mu_S)(A_\epsilon)>1-\epsilon$ be 
from Definition \ref{def:LB}.
Take 
$N\geq \max(N_0,N_\epsilon)$. Notice that for every $(x,y)\in (B\times C)\cap 
A_\epsilon$, we have 
\be\label{eq:emt}
(D_{x}\times D_y)\cap A_\epsilon=\emptyset.
\ee
Indeed, if the set above is non-empty, then there exists $(x,y)\in (B\times 
C)\cap A_\epsilon$ and $(x',y')\in (D_x\times  D_y)\cap A_\epsilon$, but then by 
\eqref{flarge} and \eqref{eq:leqeps}
$$
1/100\leq \bar{f}\left(\cP^N_0(x,y),\cP^N_0(x',y')\right)\leq \epsilon
$$
and this is a contradiction. So \eqref{eq:emt} holds. This in turn contradicts 
$\mu_T\times\mu_S(B\times C),\mu_T\times\mu_S(D_x\times D_y)\geq 99/100$ and 
$\mu_T\times\mu_S (A_\eps)\geq 1-\epsilon$. So the proof is finished up to 
proving the \textbf{Claim}.

To prove the \textbf{Claim}, we will show that for every matching $\theta=(i_s,j_s)_{s=1}^r$ of $\cP_0^N(x,y)$ and $\cP_0^N(x',y')$, we have $r\leq \frac{9N}{10}$. Fix any such matching $\theta=(i_s,j_s)_{s=1}^r$ and let
$$
H_N:=\Bigl\{s\in\{1,\ldots,r\}\;:\;  (T^{i_s}x,S^{i_s}y), (T^{j_s}x',S^{j_s}y')\in 
(\cT_{n_0}^T\cap F^T)\times ( \cT_{n_0}^S)\cap F^S) \Bigr\}.
$$
By  $(P1)$ we have $|\{1,\ldots,r\}\setminus H_n|\leq \frac{1}{10}N$, hence the
\textbf{Claim} follows by showing that
\begin{equation}\label{unvn}
|H_N|\leq \frac{7}{10}N.
\end{equation}
Notice that by definition of $A_{\theta}^{k}((x,y)(x',y'))$,  we 
have 
$$
H_N=\bigcup_{k=0}^{+\infty}A_{\theta}^{k}((x,y)(x',y')).
$$
Therefore, using \eqref{Aempty} and remembering that $n_0\ge 100$, we get  
$$
|H_N|\leq \sum_{k\geq n_0/2}\left|A_{\theta}^{k}((x,y)(x',y'))\right|
\le \sum_{k\ge 50}\left|A_{\theta}^{k}((x,y)(x',y'))\right|.
$$ 
By $(P2)$, this implies that 
$$
|H_N|\leq N \sum_{k\ge 50}\frac{1}{k^2}\leq \frac{7}{10}N.
$$
This shows \eqref{unvn}, which concludes the proof of the 
\textbf{Claim} and the proof of Theorem \ref{main2}.
\end{proof}

\section{Proof of Proposition \ref{prop:LB}}
This section will be devoted to the proof of Proposition \ref{prop:LB}. 
We will divide the proof into several steps. 

\subsection{Construction of $B$ and $C$ in Proposition \ref{prop:LB}}
Recall that  $(T,S)\in \cL^{\gamma,\gamma',\delta}$.  
{Set $\xi:=\min(\gamma,1-\gamma',\delta)/100$.} To define the sets $B$ and 
$C$, we will use some notation from Section \ref{sec:distest}.
First let $n_0$ be large enough 
so that
$$
\mu_G(\cT^G_{n_0})\geq 1-2^{-100}\;\;\text{ for } G\in \{T,S\}.
$$

Since the construction follows similar lines for $T$ and $S$ we will do it 
simultaneously for $G\in \{T,S\}$. First recall the definition of $F^G$ from~\eqref{fg}.
\smallskip
By the ergodic theorem, there exists a set $F_{erg}^G$, 
$\mu(F_{erg}^G)>1-10^{-4}$ and a number $m_3\in \N$ such that for every $x\in 
F_{erg}^G$ and every $N\geq m_3$,
\begin{equation}\label{erg:th}
\Bigl|\left\{j\in[0,N-1]\;:\;G^jx\in F^G\cap\cT_{n_0}^G \right\}\Bigr|\geq (1-10^{-4})N.
\end{equation}
We then define $B:=F^T\cap F_{erg}^T$ and $C:=F^S\cap F_{erg}^S$. We will write this in the 
product form:
\begin{equation}\label{bc}
B\times C=(F^T\cap F_{erg}^T)\times (F^S\cap F_{erg}^S).
\end{equation} 
Notice in particular that $(\mu_T\times \mu_S)(B\times C)\geq 99/100$ as 
required in Proposition \ref{prop:LB}.

\paragraph{Construction of $D_x$ and $D_y$.}
 We define 
\begin{equation}\label{def:dx}
D_x\times D_y:=(D^T_x\cap B)\times (D^S_y\cap C),
\end{equation}
where $D^T_x$ and $D^S_y$ come from \eqref{dx} for $G=T$ and $G=S$ respectively 
(see also \eqref{dgn}). Notice that $(\mu_T\times \mu_S)(D_x\times D_y)\geq 
99/100$ as in the statement of Proposition \ref{prop:LB}

\paragraph{Definition of $N_0$ and proof of $(P1)$ in Proposition \ref{prop:LB}}

Let $N_0:=\max\{m_3,m_2,m_1, n_3,n_1\}$, where $m_3$ comes from~\eqref{erg:th}, $m_2$ is from 
Lemma \ref{lem:2}, $m_1$ is from Lemma \ref{lem:1}, $n_1$ 
from \eqref{fg} and $n_3$ from \eqref{dx}\footnote{All the constants come in two 
copies: for $T$ and $S$, for instance we have $m_3^T$ and $m_3^S$ and we define 
$m_3:=\max(m_3^T,m_3^S)$. We define all other constants analogously.}.
\smallskip
Notice now that for $(x,y)\in B\times C$ and $(x',y')\in D_x\times D_y\subset 
B\times C$ (see \eqref{def:dx}) and for $N\geq N_0\geq m_2$, 
$(P1)$ holds since by \eqref{erg:th}, for any $(z,w)\in B\times C$,
$$\left|\{i\in [0,N]\;:\; (T^iz,S^iw)\in ({F^T}\cap \cT^T_{n_0})\times ({F^S}\cap \cT^S_{n_0})\}\right| 
\geq \left(1-\frac{2}{10^4}\right)N.
$$
%for $(z,w)\in \{(x,y),(x',y')\}$. Therefore we only have to prove $(P2)$.

\subsection{Proof of $(P2)$ in Proposition \ref{prop:LB}}
\begin{proof}
We will use Lemma \ref{lem:1}, Lemma \ref{lem:2} for $G=T$ and $G=S$ and then 
Lemma \ref{comb}. Fix $N > N_0$,  $(x,y)\in B\times C$ and $x',y'\in 
D_x\times D_y$. Consider any matching $\theta=(i_s,j_s)_{s=1}^r$ between $(\cP_{n_0})_0^{N}(x,y)$ and 
$(\cP_{n_0})_0^{N}(x',y')$, and let $k\in\N$.
If $k<n_0/2$, then $A_{\theta}^{k}((x,y),(x',y'))=\emptyset$ (see 
\eqref{Aempty}) and $(P2)$ follows trivially. We therefore assume that 
$k\ge n_0/2$.  

Take any $s\in \{1,...,r\}\cap A_{\theta}^{k}((x,y),(x',y'))$.  To simplify 
notation, denote $x_s=T^{i_s}x, x'_s=T^{j_s}x'$ and $y_s=S^{i_s}y, 
y'_s=S^{j_s}y'$. By definition of $A_{\theta}^{k}((x,y),(x',y'))$, we have
\be\label{eq:jk0}
2^{-k-1}\leq \max\left(d_1(x_s,x'_s),d_2(y_s,y'_s)\right)<2^{-k}.
\ee
Let $n,m\in \N$ be unique such that $(h_n^T)^{-1}=d_1(x_s,x'_s)$ and 
$(h_m^S)^{-1}=d_2(y_s,y'_s)$.  
We assume WLOG that $h_n^T < h_m^S$ ({notice that, by the $\delta$-alternation, we cannot have $h_n^T=h_m^S$,} and the case %$2^k< h_m^S\leq 2^{k+1}$ 
$h_n^T > h_m^S$ is analogous). Then by~\eqref{eq:jk0}, we have
\begin{equation}\label{h2}
2^k<h_n^T\leq 2^{k+1},
\end{equation}
 Then, again by the 
$\delta$-alternation and \eqref{eq:jk0} we know that
\begin{equation}\label{smallhm}
h_m^S\geq 2^{k(1+\frac{\delta}{2})}.
\end{equation}
%Recall the definition of $\xi$ from Lemma~\ref{lem:1}.
We will first show that $2^{k(1+2\xi)+3}<N$. Indeed,  otherwise 
we would have $N\le 2^{k(1+\delta/50)+3}$ (recall that $\xi\le\delta/100$).
Then  by \eqref{smallhm}
there would exist some $\eta=\eta(\delta)>0$ such that $h_m^S\geq N^{1+\eta}$; in particular, $\frac{h_m^S}{m^2}\geq N$ ($m$ 
is sufficiently large). 
By an application of Lemma \ref{lem:2} with $y\in F^S$, $y'\in D_y^S$ and $S^{i_s}y\in F^G\subset \cT_m$, 
and since $0\leq i_s,j_s\leq N\leq \frac{h_m^S}{m^2}$, we would get
$$
\frac{1}{h_m^S}=d_2(y_s,y'_s)=d_2(S^{i_s}y,S^{j_s}y')\geq \frac{1}{h_{m-1}^S},
$$
a contradiction. Hence 
\be\label{eq:nolo}
2^{k(1+2\xi)+3}<N.
\ee

Let $v=v(n)$ be unique such that 
$$
h_v^S\leq h_n^T<h_{v+1}^S.
$$
By the $\delta$-alternation, since $n$ is large enough and $\xi\leq 
\delta/100$, we know that 
\be\label{eq:hesep}
(h_n^T)^{1+2\xi}\leq\frac{h_{v+1}^S}{(v+1)^3}\;\;\text{ and } \;\;
(h_{v}^S)^{1+2\xi}\leq h_n^T.
\ee 
Notice that by definition of $v$ and since $h_n^T<h_m^S$,  we have $m\geq v+1$ and therefore 
 $d_2(y_s,y'_s)=(h_m^S)^{-1}\leq (h_{v+1}^S)^{-1}$. Moreover,  $y_s,y_s'\in F^S$ 
(see \eqref{eq:ats}). Then, by~\eqref{unif:dist} (for $n=v+1$, $y_s$, $y'_s$ and 
$G=S$), for every $0\leq i,j\leq (h_n^T)^{1+2\xi}<\frac{h_{v+1}^S}{(v+1)^3}$, 
$i\neq j$, we get
\begin{equation}\label{eq12}
d_2(S^i(y_s),S^j(y'_s))\geq (h^S_{v})^{-1}\stackrel{\eqref{eq:hesep}}{\geq} 
(h_n^T)^{-\frac{1}{1+2\xi}}\stackrel{\eqref{h2}}{\geq} 2^{-k+1}.
\end{equation}
{(For the last inequality, we also used the fact that $k\ge n_0/2$ is large enough.) }
Hence, if for some $w>s$, we have
  $$(i_w,j_w)\in [i_s,i_s+(h_n^T)^{1+2\xi}]\times[j_s,j_s+(h_n^T)^{1+2\xi}]\cap 
A_{\theta}^{k}((x,y),(x',y')),$$ 
then 
$i_w-i_s=j_w-j_s$. Indeed, if not then by \eqref{eq12} for $i=i_w-i_s$ and $j=j_w-j_s$
$$
d_2(y_w,y'_w)=d_2(S^{i_w-i_s}(y_s),S^{j_w-j_s}(y'_s))\geq 2^{-k+1}.
$$
This would contradict the definition 
of $A_{\theta}^{k}((x,y),(x',y'))$.
\smallskip
So $i_w-i_s=j_w-j_s$. Let $w>s$ be such that 
$i_w-i_s=j_w-j_s\in [h_n^T,(h_n^T)^{1+2\xi}]$. By \eqref{short:bl} (for $G=T$ 
and $x_s,x_s'\in F^T$) and since $x_w,x'_w\in F^T$ (see \eqref{eq:ats}) we get 
for $r_w:=i_w-i_s=j_w-j_s$,
\begin{equation}\label{eq23}
d_1(x_w,x'_w)=d_1(T^{r_w}(x_s),T^{r_w}(x'_s)) \geq 
(h^T_{n-1})^{-1}\stackrel{\eqref{h2}}{\geq}2^{-k+1}. 
\end{equation}
{The last inequality holds because $T\in C_{\gamma,\gamma'}$, hence $(h_{n-1}^T)^{-1}\ge (h_n^T)^{-\frac{1}{1+\gamma}} \ge (h_n^T)^{-\frac{1}{1+2\xi}}$ and we remember the last inequality in~\eqref{eq12}.}
Therefore, if $w>s$ is such that 
$$(i_w,j_w)\in [i_s+h_n^T,i_s+(h_n^T)^{1+2\xi}]\times[j_s+h_n^T,j_s+(h_n^T)^{1+2\xi}] $$
%
%$r_w\in [h_n^T,(h_n^T)^{1+2\xi}]$, 
then by 
\eqref{eq23}, $w\notin A_{\theta}^{k}((x,y)(x',y'))$.

By \eqref{h2} it follows that if  
$$ (i_w,j_w)\in 
[i_s,i_s+2^{j(1+2\xi)}]\times[j_s,j_s+2^{j(1+2\xi)}]\subset[i_s ,i_s+(h_n^T)^{1+2\xi}]\times[j_s ,j_s+(h_n^T)^{1+2\xi}] $$ 
is such that $w\in A_{\theta}^{k}((x,y),(x',y'))$, then   
$$
i_w-i_s=j_w-j_s < h_n^T \le 2^{k+1}.
$$ 
Let $\{w_s\}_{s=1}^{r'}:=A_{\theta}^{k}((x,y),(x',y'))\cap 
\{1,\ldots,r\}$ and consider the matching subsequence given by 
$\{i_{w_s},j_{w_s}\}_{s=1}^{r'}$. 

For $0\le M\le N$ and $w\in \{1,\ldots,r'\}$, consider the sets $I(M,w)\subset\{1,\ldots r'\}$ and $J(M,w)\subset\{1,\ldots r'\}$ defined as in Section~\ref{sec:BegEnd}. Then by the above reasoning, it follows that for $K=2^k$,  we have
$$
\min\left(|I(K^{1+\xi},s)|,|J(K^{1+\xi},s)|\right)\leq 2K.
$$
Since \eqref{eq:nolo} holds, the assumptions of Lemma \ref{comb} are satisfied, and we get 
$$
\left|A_{\theta}^{k}((x,y)(x',y'))\right|\leq \frac{4N}{2^{k\xi}}\leq 
\frac{N}{k^2}.
$$
(The last inequality since $k\ge n_0/2$ is sufficiently large.)
This finishes the proof of $(P2)$ and therefore also the proof of Proposition 
\ref{prop:LB}
\end{proof}

\bibliography{nonLB}

% 
% 
% \begin{thebibliography}{9}
% \bibitem{Fel} J.\ Feldman, {\it New $K$-automorphisms and a problem of 
% Kakutani}. Israel J. Math. 24.1 (1976): 16-38.
% \bibitem{KaWe} A.\ Kanigowski, D.\ Wei, {\it Product of two Kochergin flows with 
% different exponent is not standard}, accepted in Studia Mathematica.
% \bibitem{Kat0} A.\ B.\ Katok, {\it Monotone equivalence in ergodic theory}, 
% (Russian) Izv. Akad. Nauk SSSR Ser. Mat. 41 (1977), no. 1, 104--157.
% \bibitem{Kat1} A.\ B.\ Katok, {\it Combinatorical Constructions in Ergodic 
% Theory and Dynamics}, American Mathematical Society, Providence, 2003.
% \bibitem{KatSto} A. Katok and A. Stepin, {\it Approximations in ergodic theory}, 
% Russ. Math. Sur- veys, 22: 77 - 102, 1967.
% \bibitem{KuGe} P.\ Kunde, M.\ Gerber, {\it A smooth zero-entropy diffeomorphism 
% whose product with itself is loosely Bernoulli}, arXiv:1803.01926 
% \bibitem{OrRuWe} D.\ Ornstein, D.\ Rudolph, B.\ Weiss, {\it Equivalence of 
% measure preserving transformations}, Mem. Amer. Math. Soc., 37(262), 1982.
% \bibitem{Ratner3} M. Ratner, {\it Horocycle flows are loosely Bernoulli}. Israel 
% J. Math. (1978), 31: 122-132.
% \bibitem{Rat1} M. Ratner, {\it The Cartesian square of the horocycle flow is not 
% loosely Bernoulli}. Israel J. Math. 34, no. 1-2, 72--96 (1980).
% \end{thebibliography}
% 
% 

\end{document}